\documentclass[letterpaper,10pt]{article}

\usepackage[letterpaper,left=1in,right=1in,top=1.5in,bottom=1.5in]{geometry}
\newcommand{\myauthor}{Benjamin Antieau}
\newcommand{\mytitle}{Periodic cyclic homology and derived de Rham cohomology}

\title{\mytitle}

\author{Benjamin Antieau}
\date{\today}
 
\usepackage[pdfstartview=FitH,
            pdfauthor={\myauthor},
            pdftitle={\mytitle},
            colorlinks,
            linkcolor=reference,
            citecolor=citation,
            urlcolor=e-mail,
            ]{hyperref}

\usepackage{amsmath}
\usepackage{amscd}
\usepackage{amsbsy}
\usepackage{amssymb}
\usepackage{verbatim}
\usepackage{eufrak}
\usepackage{eucal}
\usepackage{microtype}
\usepackage{hyperref}
\usepackage{mathrsfs}
\usepackage{amsthm}
\usepackage{stmaryrd}
\usepackage[all,cmtip]{xy}
\usepackage{tikz}
\usetikzlibrary{matrix,arrows}
\usepackage[abbrev,lite,alphabetic]{amsrefs}
\usepackage{dsfont}

\usepackage{fancyhdr}
\usepackage{makeidx}
\makeindex

\usepackage{color}
\definecolor{todo}{rgb}{1,0,0}
\definecolor{conditional}{rgb}{0,1,0}
\definecolor{e-mail}{rgb}{0,.40,.80}
\definecolor{reference}{rgb}{.20,.60,.22}
\definecolor{mrnumber}{rgb}{.80,.40,0}
\definecolor{citation}{rgb}{0,.40,.80}

\let\oldmarginpar\marginpar
\renewcommand\marginpar[1]{\-\oldmarginpar[\raggedleft\footnotesize #1]%
{\raggedright\footnotesize #1}}

\newcommand{\Cscr}{\mathcal{C}}

\newcommand{\B}{\mathrm{B}}
\newcommand{\C}{\mathrm{C}}
\newcommand{\D}{\mathrm{D}}
\newcommand{\E}{\mathrm{E}}
\newcommand{\F}{\mathrm{F}}
\renewcommand{\H}{\mathrm{H}}

\renewcommand{\L}{\mathrm{L}}

\renewcommand{\P}{\mathrm{P}}
\newcommand{\R}{\mathrm{R}}

\newcommand{\W}{\mathrm{W}}

\newcommand{\CC}{\mathds{C}}
\newcommand{\EE}{\mathds{E}}

\newcommand{\NN}{\mathds{N}}
\newcommand{\PP}{\mathds{P}}
\newcommand{\QQ}{\mathds{Q}}

\newcommand{\ZZ}{\mathds{Z}}

\newcommand{\gr}{\mathrm{gr}}

\newcommand{\crys}{\mathrm{crys}}

\newcommand{\op}{\mathrm{op}}

\newcommand{\heart}{\heartsuit}

\newcommand{\fib}{\mathrm{fib}}

\renewcommand{\geq}{\geqslant}
\renewcommand{\leq}{\leqslant}

\newcommand{\HC}{\mathrm{HC}}

\newcommand{\HP}{\mathrm{HP}}
\newcommand{\TP}{\mathrm{TP}}
\newcommand{\TC}{\mathrm{TC}}

\newcommand{\HH}{\mathrm{HH}}

\DeclareMathOperator{\Fun}{Fun}

\newcommand{\Ch}{\mathrm{Ch}}

\newcommand{\sCAlg}{\mathrm{sCAlg}}
\newcommand{\CAlg}{\mathrm{CAlg}}

\DeclareMathOperator*{\colim}{colim}

\DeclareMathOperator{\Spec}{Spec}

\newcommand{\we}{\simeq}
\newcommand{\iso}{\cong}

\theoremstyle{plain}
\newtheorem{theorem}{Theorem}[section]
\newtheorem*{theorem*}{Theorem}
\newtheorem{lemma}[theorem]{Lemma}

\newtheorem{corollary}[theorem]{Corollary}
\newtheorem*{corollary*}{Corollary}

\theoremstyle{plain}

\theoremstyle{definition}

\newtheoremstyle{named}{}{}{\itshape}{}{\bfseries}{.}{.5em}{#1 \thmnote{#3}}
\theoremstyle{named}

\theoremstyle{definition}
\newtheorem{definition}[theorem]{Definition}

\newtheorem{example}[theorem]{Example}
\newtheorem*{example*}{Example}

\newtheorem*{question*}{Question}

\newtheorem{remark}[theorem]{Remark}

\begin{document}

\maketitle

\begin{abstract}
    \noindent
    We use the Beilinson $t$-structure on filtered complexes
    and the Hochschild--Kostant--Rosenberg theorem to
    construct filtrations on the negative cyclic and periodic cyclic homologies of a scheme $X$ with
    graded pieces given by the Hodge-completion of the derived de Rham
    cohomology of $X$. Such filtrations have
    previously been constructed by
    Loday in characteristic zero and by Bhatt--Morrow--Scholze for $p$-complete
    negative cyclic and periodic cyclic homology in the quasisyntomic case.


    \paragraph{Mathematics Subject Classification 2010.}
    \href{http://www.ams.org/mathscinet/msc/msc2010.html?t=13Dxx&btn=Current}{13D03}
    --
    \href{http://www.ams.org/mathscinet/msc/msc2010.html?t=13Dxx&btn=Current}{14F40}.
\end{abstract}


\section{Introduction}

\newcommand{\BMS}{\mathrm{BMS}}
\newcommand{\DF}{\mathrm{DF}}
\newcommand{\HKR}{\mathrm{HKR}}
\newcommand{\dHH}{\mathrm{dHH}}
\newcommand{\dHC}{\mathrm{dHC}}
\newcommand{\dHP}{\mathrm{dHP}}
\newcommand{\DBF}{\mathrm{DBF}}
\newcommand{\CW}{\mathrm{CW}}

Let $k$ be a quasisyntomic ring and let $k\rightarrow R$ be a
quasisyntomic $k$-algebra. Bhatt, Morrow, and Scholze construct in~\cite{bms2}*{Theorem~1.17} a
functorial complete exhaustive decreasing multiplicative $\ZZ$-indexed filtration
$\F^\star_{\BMS}\HP(R/k;\ZZ_p)$ on the $p$-adic completion
$\HP(R/k;\ZZ_p)$ of periodic cyclic homology with graded pieces
$\gr_{\BMS}^n\HP(-/k;\ZZ_p)\we\overline{\L\Omega}_{R/k}[2n]$, where $\L\Omega_{R/k}$ is the derived de Rham
complex and $\overline{\L\Omega}_{R/k}$ is the $p$-adic completion
of the Hodge-completion of this complex. The Hodge filtration $\Omega^{\geq
n}_{R/k}$ for smooth algebras induces a Hodge filtration $\L\Omega_{R/k}^{\geq
n}$ on the derived de Rham complex and its completed variants.
There is a corresponding filtration on
negative cyclic homology, with graded pieces given by
$\overline{\L\Omega}_{R/k}^{\geq n}[2n]$, the $p$-completion of the
Hodge-completion of $\L\Omega_{R/k}^{\geq n}[2n]$.

For smooth $\QQ$-algebras, a similar statement goes back to
Loday~\cite{loday}*{5.1.12}. One can also derive very general results along
these lines in characteristic zero from~\cite{toen-vezzosi-s1}.
The authors of~\cite{bms2} suggest that such a filtration should exist outside the
$p$-complete setting. In this note, we use the Beilinson $t$-structure on filtered
complexes~\cite{beilinson-perverse} to prove that this is indeed the case.

\begin{theorem}\label{thm:intro}
    If $k$ is a commutative ring and $X$ is a quasi-compact quasi-separated $k$-scheme, then there are
    functorial complete decreasing multiplicative $\ZZ$-indexed filtrations
    $\F^\star_\B\HC^-(X/k)$ and $\F^\star_\B\HP(X/k)$ on negative cyclic
    homology and periodic cyclic homology, respectively. These filtrations
    satisfy the following properties.
    \begin{enumerate}
        \item[{\rm (a)}] There are natural equivalences
            \begin{align*}
                \gr^n_\B\HC^-(X/k)&\we\R\Gamma(X,\widehat{\L\Omega}^{\geq
            n}_{-/k}[2n]),\\
            \gr^n_\B\HP(X/k)&\we\R\Gamma(X,\widehat{\L\Omega}_{-/k}[2n]),
            \end{align*}
            where $\widehat{\L\Omega}_{-/k}$ is the Hodge-completion of the derived
            de Rham complex and $\widehat{\L\Omega}^{\geq n}_{-/k}$ is $n$th term in
            the Hodge filtration on $\widehat{\L\Omega}_{-/k}$.
        \item[{\rm (b)}] The filtered pieces $\F^n_\B\HC^-(X/k)$ and
            $\F^n_\B\HP(X/k)$ are equipped with compatible decreasing
            filtrations which induce the Hodge filtration on
            $\gr^n_\B\HC^-(X/k)$ and $\gr^n_\B\HP(X/k)$
            under the equivalences of part (a).
        \item[{\rm (c)}] If $X/k$ is quasi-lci,\footnote{We say that a $k$-scheme $X$ is {\bf quasi-lci} if $\L_{X/k}$ has Tor-amplitude
            contained in $[0,1]$.} then the filtrations
            $\F^\star_\B\HC^-(X/k)$ and $\F^\star_\B\HP(X/k)$ are exhaustive.
    \end{enumerate}
\end{theorem}

Negative cyclic homology and periodic cyclic homology satisfy fpqc descent
by~\cite{bms2}*{Corollary~3.3} as a consequence
of the fact that the derived exterior powers $\Lambda^i\L_{-/k}$ of the
cotangent complex are fpqc sheaves by \cite{bms2}*{Theorem~3.1}. Since
$\widehat{\L\Omega}^{\geq n}_{-/k}$ has by definition a complete exhaustive decreasing
$\NN$-indexed filtration with graded pieces $\Lambda^i\L_{-/k}$, it follows
that the Hodge-truncated Hodge-completed derived de Rham complexes
$\widehat{\L\Omega}^{\geq n}_{-/k}$ are also fpqc sheaves. Thus, to prove the
theorem, it suffices to handle the affine case.

Theorem~\ref{thm:intro} follows from a much more general theorem,
Theorem~\ref{thm:bi}, which states
that in a suitable $\infty$-category, of bicomplete bifiltered complexes, the
Beilinson filtrations are exhaustive for any quasi-compact quasi-separated
$k$-scheme $X$.

\begin{remark}
    \begin{enumerate}
        \item[(i)] In case both are defined, the $p$-adic completion of the filtration of
            Theorem~\ref{thm:intro} agrees with the filtration
            of~\cite{bms2}*{Theorem~1.17}. This follows in the smooth case by examining
            the proofs of each theorem and in general by mapping the left Kan extension
            of our proof to the filtration obtained by quasisyntomic descent in their
            proof.
        \item[(ii)] In~\cite{an1}, with Nikolaus, we introduce a $t$-structure
            on cyclotomic spectra. As one application of the $t$-structure, we show using
            calculations of Hesselholt~\cite{hesselholt-ptypical}, that the methods of this paper
            can be used to construct a filtration $\F^\star_\B\TP(X)$ on
            topological periodic cyclic homology $\TP(X)$ when $X$ is a
            smooth scheme over a perfect field with graded pieces given by
            (shifted) crystalline cohomology $\gr^n_\B\TP(X)\we\R\Gamma_{\crys}(X/W(k))[2n]$.
            When $X=\Spec R$ is smooth and affine, then in fact
            $\gr^n_\B\TP(X)$ is given canonically by $\W\Omega^\bullet_R[2n]$,
            the shifted de Rham--Witt complex. This
            recovers several parts of~\cite{bms2}*{Theorems~1.10, 1.12,
            and~1.15} in the case of a smooth scheme over a perfect field.
    \end{enumerate}
\end{remark}

\paragraph{Outline.} In Section~\ref{sec:filtrations}, we outline the theory of filtrations we
will need. We explain the smooth affine case in Section~\ref{sec:smooth}.
In Section~\ref{sec:general}, we give the full proof, which follows from the
smooth case by taking non-abelian derived functors in an appropriate
$\infty$-category of bifiltrations.

\paragraph{Conventions and notation.} We work with $\infty$-categories
throughout, following the conventions of~\cite{htt} and~\cite{ha}. Hochschild
homology $\HH(R/k)$ and its relatives are viewed as objects in the derived
$\infty$-category $\D(k)$, possibly with additional structure. Typically, we
view objects of $\D(k)$ as being given by chain complexes up to
quasi-isomorphism, but several constructions will lead us to cochain complexes
as well. Given an object $X\in\D(k)$, we will write $\H_*X$ for its homology groups. We will write
$X^\bullet$ for a given cochain complex model for $X$. Thus, $X^\bullet$ is an
object of the category $\Ch(k)$ of cochain complexes of $k$-modules. The main example is the de Rham complex $\Omega_{R/k}^\bullet$ for a smooth commutative $k$-algebra
$R$.

\paragraph{Acknowledgments.} We thank Thomas Nikolaus for explaining the
interaction of the $S^1$-action and the Hopf element $\eta$ and Peter Scholze
for bringing this problem to our attention. We also benefited from
conversations with Dmitry Kaledin and Akhil Mathew about the Beilinson
$t$-structure. Finally, Elden Elmanto generously provided detailed comments on
a draft of the paper. This work was supported by NSF Grant DMS-1552766.

\section{Background on filtrations}\label{sec:filtrations}

Throughout this section, fix a commutative ring $k$. Let $\D(k)$ be the derived
$\infty$-category of $k$, a stable presentable $\infty$-categorical enhancement
of derived category of unbounded chain complexes of $k$-modules.
The derived tensor product of chain complexes makes $\D(k)$ into a presentably symmetric
monoidal stable $\infty$-category, meaning that $\D(k)$ is a symmetric monoidal
presentable $\infty$-category in which the tensor product commutes
with colimits in each variable.

The {\bf filtered derived $\infty$-category} of $k$ is $\DF(k)=\Fun(\ZZ^\op,\D(k))$,
the $\infty$-category of sequences $$X(\star)\colon\cdots\rightarrow X(n+1)\rightarrow
X(n)\rightarrow\cdots$$ in $\D(k)$. Write $X(\infty)=\lim_n
X(n)\we\lim\left(\cdots\rightarrow X(n+1)\rightarrow
X(n)\rightarrow\cdots\right)$ for the limit of the filtration.
A filtered complex $X(\star)\in\DF(k)$
is {\bf complete} if $X(\infty)\we 0$. Similarly, write $X(-\infty)$ for
$\colim_n X(n)\we\colim\left(\cdots\rightarrow X(n+1)\rightarrow
X(n)\rightarrow\cdots\right)$. Given a map
$X(-\infty)\rightarrow Y$, we say that $X(\star)$ is a
filtration on $Y$; if the map is an equivalence, we say that $X(\star)$ is an
{\bf exhaustive} filtration on $Y$.

We will refer to general objects $X(\star)$ of $\DF(k)$ as {\bf decreasing
$\ZZ$-indexed filtrations}.
We will write $\gr^nX$ for the cofiber of $X(n+1)\rightarrow X(n)$, the $n$th
graded piece of the filtration. Several filtrations of interest in this paper are in fact
{\bf $\NN$-indexed}, meaning that $X(0)\we X(-1)\we X(-2)\we\cdots$, or equivalently
that $\gr^n X\we 0$ for $n<0$.

Day convolution (using the additive symmetric monoidal structure of $\ZZ^\op$) makes $\DF(k)$ into
a presentably symmetric monoidal stable $\infty$-category. The Day convolution symmetric monoidal
structure has the property that
if $X(\star)$ and $Y(\star)$ are filtered objects of $\D(k)$, then $(X\otimes_kY)(\star)$
is a filtered spectrum with graded pieces
$\gr^n(X\otimes_kY)\we\bigoplus_{i+j=n}\gr^iX\otimes_k\gr^jY$.

A filtration $X(\star)$ equipped with the structure of a commutative algebra object (or
$\EE_\infty$-algebra object) in $\DF(k)$ is called a {\bf multiplicative} filtration.

One source of decreasing filtrations is via the Whitehead tower with respect to some
$t$-structure on $\D(k)$. We will use the standard Postnikov $t$-structure,
which has $\D(k)_{\geq 0}\subseteq\D(k)$ the full subcategory of $X$ such that
$\H_n(X)=0$ for $n<0$. Similarly, $\D(k)_{\leq 0}$ is the full subcategory of
$\D(k)$ consisting of $X$ such that $\H_n(X)=0$ for $n>0$. Given an object $X$,
its $n$-connective cover $\tau_{\geq n}X\rightarrow X$ has $\H_i(\tau_{\geq
n}X)\iso\H_i(X)$ for $i\geq n$ and $\H_i(\tau_{\geq n}X)=0$ for $i<n$.

\begin{example}
    If $R$ is a connective commutative algebra object in $\D(k)$, then the
    Whitehead tower $\tau_{\geq\star}R$ is a complete exhaustive decreasing multiplicative
    $\NN$-indexed filtration on $R$ with $\gr^n\tau_{\geq\star}R\we\H_n(R)[n]$.
\end{example}

For details and proofs of the statements above, see~\cite{gwilliam-pavlov}. For
more background, see~\cite{bms2}*{Section 5}. Now, we introduce the Beilinson
$t$-structure on $\DF(k)$.

\begin{definition}
    Let $\DF(k)_{\geq i}\subseteq\DF(k)$ be the full subcategory of those
    filtered objects $X(\star)$ such that $\gr^nX\in\D(k)_{\geq i-n}$.
    We let $\DF(k)_{\leq i}\subseteq\DF(k)$ be the full subcategory of those
    filtered objects $X(\star)$ such that $X(n)\in\D(k)_{\leq i-n}$.
\end{definition}

Note the asymmetry in the definition. The pair
$(\DF(k)_{\geq 0},\DF(k)_{\leq 0})$ defines a $t$-structure on $\DF(k)$
by~\cite{beilinson-perverse}; see also~\cite{bms2}*{Theorem~5.4} for a proof. We will
write $\tau_{\leq n}^\B$, $\tau_{\geq n}^\B$, $\pi_n^\B$ for the truncation and
homotopy object functors in the Beilinson $t$-structure.

The connective objects $\DF(k)_{\geq 0}$ are closed under the tensor
product on $\DF(k)$, and hence the natural map $\pi_0^\B\colon\DF(k)_{\geq
0}\rightarrow\DF(k)^\heart$ is symmetric monoidal. The heart is the abelian
category of cochain complexes of $k$-modules equipped with the usual tensor
product of cochain complexes.

\begin{remark}\label{rem:underlying}
    The Beilinson Whitehead tower $\tau_{\geq\star}^\B X$ is most naturally a
    bifiltered object, since each $\tau_{\geq n+1}^\B X\rightarrow\tau_{\geq n}^\B
    X$ is a map of objects of $\DF(k)$. If we forget the residual
    filtration on $\tau_{\geq\star}^\B X$ (by taking the colimit), then we obtain a
    new filtration on $X(-\infty)$.
    In this paper, we will need this only for $\NN$-indexed filtrations. In this
    case, each $n$-connective cover $\tau_{\geq n}^\B X$ is also $\NN$-indexed, and
    we can view the resulting filtration $(\tau_{\geq n}^\B X)(0)$ as a new filtration
    on $X(0)$.\footnote{Note that this is not an idempotent operation: applying
    the Beilinson Whitehead tower to $\tau_{\geq\star}^\B X(0)$ typically
    produces yet another filtration on $X(0)$.} If $X$ is a commutative algebra object of $\DF(k)$, then the Beilinson
    Whitehead tower $\tau_{\geq\star}^\B X$ is a new multiplicative filtration on
    $X$.
\end{remark}

For our purposes, it will be most important to understand the $n$-connective
cover functors. Given $X(\star)\in\DF(k)$, the $n$-connective cover in the
Beilinson $t$-structure $\tau_{\geq
n}^\B X\rightarrow X(\star)$ induces equivalences $$\gr^i\tau_{\geq n}^\B
X\we\tau_{\geq n-i}\gr^i X$$
(see~\cite{bms2}*{Theorem~5.4}). From this, we see that $\gr^i\pi_n^\B
X\we(\H_{n-i}(\gr^iX))[-i]$. The cochain complex corresponding to
$\pi_n^\B X$ is of the form
$$\cdots\rightarrow\H_n(\gr^0X)\rightarrow\H_{n-1}(\gr^1X)\rightarrow\H_{n-2}(\gr^2X)\cdots,$$
where $\H_n(\gr^0X)$ is in cohomological degree $0$ and where the differentials
are induced from the boundary maps in homology coming from the cofiber
sequences $\gr^{i+1}X\rightarrow X(i)/X(i+2)\rightarrow\gr^iX$.
See~\cite{bms2}*{Theorem~5.4(3)} for details.

The next example illustrates our main idea in a general setting.

\begin{example}\label{ex:circle}
    Let $X\in\D(k)$ be an object equipped with an $S^1$-action. The Whitehead tower
    $\tau_{\geq\star}X$ defines a complete exhaustive $S^1$-equivariant
    $\ZZ$-indexed filtration $\F^\star_\P X$ on $X$ with graded pieces
    $\gr^n_\P X\we\H_n(X)[n]$, equipped with the trivial $S^1$-action.
    Applying homotopy $S^1$-fixed points, we obtain a complete $\ZZ$-indexed
    filtration $\F^\star_\P X^{hS^1}$ on $X^{hS^1}$ with graded pieces
    $\gr^n_\P X^{hS^1}\we(\H_n(X)[n])^{hS^1}$. Let $\F^\star_\B X^{hS^1}$ be
    the double-speed Whitehead tower of $\F^\star_\P X^{hS^1}$ in the Beilinson
    $t$-structure on $\DF(k)$, so that $\F^n_\B X^{hS^1}=\tau_{\geq
    2n}^\B\F^\star X^{hS^1}$. By definition, $\F^n_\B X^{hS^1}$ is a
    filtered spectrum with $$\gr^i\F^n_\B X^{hS^1}\we\tau_{\geq 2n-i}\gr^i_\P
    X^{hS^1}\we\tau_{\geq 2n-i}(\H_i(X)[i])^{hS^1}.$$ Hence, $$\gr^i\gr^n_\B
    X^{hS^1}\we\begin{cases}\H_i(X)[2n-i]&\text{if $n\leq
    i$,}\\0&\text{otherwise.}\end{cases}$$ This shows in fact that $\gr^n_\B
    X^{hS^1}[-2n]\we\pi_{2n}^\B\F^\star_\P X^{hS^1}$ and hence is in
    $\DF(k)^\heart$, the abelian category of cochain complexes, and is
    represented by a cochain complex
    $$0\rightarrow\H_n(X)\rightarrow\H_{n+1}(X)\rightarrow\H_{n+2}(X)\rightarrow\cdots,$$
    where $\H_n(X)$ is in cohomological degree $n$. The differential is given
    by Connes--Tsygan $B$-operator. An object $X\in\D(k)$ with an
    $S^1$-action is the same as a dg module over $\C_\bullet(S^1,k)$, the dg
    algebra of chains on
    $S^1$. The fundamental class $B$ of the circle defines a $k$-module
    generator of $\H_1(S^1,k)$ and $B^2=0$. The differential in the
    cochain complex above is given by the action of $B$. Hence, we have
    obtained a filtration $\F^\star_\B X^{hS^1}$ on $X^{hS^1}$ with graded
    pieces given by $(\H_{\bullet\geq n}(X),B)[2n]$.
\end{example}

\begin{remark}
    The same argument shows that there is a filtration $\F^\star_\B X^{tS^1}$
    on the $S^1$-Tate construction $X^{tS^1}$ with
    graded pieces $\gr^n_\B X^{tS^1}\we(\H_\bullet(X),B)[2n]$.
    We ignore for the time being any convergence issues.
\end{remark}

\section{The smooth case}\label{sec:smooth}

The Hochschild--Kostant--Rosenberg theorem~\cite{hochschild-kostant-rosenberg}
implies that there are canonical isomorphisms $\Omega^n_{R/k}\iso\HH_n(R/k)$
when $R$ is a smooth commutative $k$-algebra. In particular, letting
$\F^\star_{\HKR}\HH(R/k)$ denote the usual Whitehead tower, given by the good
truncations $\tau_{\geq\star}\HH(R/k)$,  we see that there are natural
equivalences $\gr^n_{\HKR}\HH(R/k)\we\Omega^n_{R/k}[n]$ for all $n\geq 0$.
Applying homotopy $S^1$-fixed points, we obtain a complete exhaustive
decreasing multiplicative $\NN$-indexed filtration $\F^\star_{\HKR}\HC^-(R/k)$
on $\HC^-(R/k)$.

\begin{definition}
    Let $\F^\star_\B\HC^-(R/k)$ be the double-speed Beilinson Whitehead tower for the
    filtration $\F^\star_{\HKR}\HC^-(R/k)$, so that
    $\F^n_\B\HC^-(R/k)=\tau_{\geq 2n}^\B\F^\star_\HKR\HC^-(R/k)$. For a picture
    of this filtration, see Figure~\ref{fig:ss}.
\end{definition}

Example~\ref{ex:circle} implies that this is an
multiplicative $\NN$-indexed filtration on $\HC^-(R/k)$; each
graded piece $\pi_n^\B\F^\star_{\HKR}\HC^-(R/k)\we\gr^n_\B\HC^-(R/k)[-2n]$ in
$\DF(k)^\heart$ is given by a cochain complex of the form
$$\cdots\rightarrow 0\rightarrow\Omega^{n}_{R/k}\rightarrow\Omega^{n+1}_{R/k}\rightarrow\cdots,$$ where
$\Omega^n_{R/k}$ is in cohomological degree $n$. It is verified
in~\cite{loday}*{2.3.3} that the differential is indeed the de Rham
differential. This can also be checked by hand in the case of $k[x]$ to which
the general case reduces. It follows that $\gr^n_\B\HC^-(R/k)\we\Omega^{\bullet\geq
n}_{R/k}[2n]$. The additional filtration on $\F^\star_\B\HC^-(R/k)$ reduces to the Hodge
filtration on $\Omega^{\bullet\geq n}_{R/k}[2n]$.
The exhaustiveness and completeness of $\F^\star_\B\HC^-(R/k)$ follows from
Lemma~\ref{lem:basics} below. The case of $\HP(R/k)$ is similar. This proves
Theorem~\ref{thm:intro} in the case of smooth algebras.\footnote{Note that for
$R$ a smooth $k$-algebra, the de Rham complex $\Omega^\bullet_{R/k}$ is already
Hodge-complete.}

We needed the following lemma in the proof.

\begin{lemma}\label{lem:basics}
    Let $X(\star)$ be a complete $\NN$-indexed filtration on $X=X(0)$ and let
    $\tau_{\geq\star}^\B X$ be the associated Beilinson Whitehead tower in $\DF(k)$.
    \begin{enumerate}
        \item[{\rm (i)}] The truncations $\tau_{\geq n}^\B X$ and $\tau_{\leq
            n-1}^\B X$ are complete for all $n\in\ZZ$.
        \item[{\rm (ii)}] The
            filtration $(\tau_{\geq\star}^\B X)(0)$ on $X\we X(0)$ is complete and exhaustive.
    \end{enumerate}
\end{lemma}

\begin{proof}
    Since the full subcategory $\widehat{\DF}(k)\subseteq\DF(k)$ of complete
    filtrations is stable, to prove part (i), it is enough to show that
    $\tau_{\leq n-1}^\B X$ is complete for all $n$. However, $(\tau_{\leq
    n-1}^\B X)(i)\in\D(k)_{\leq n-1-i}$. We find that
    $\lim_{i}(\tau_{\leq n-1}^\B X)(i)$ is in $\D(k)_{\leq
    -\infty}\we 0$. This proves (i). It follows from (i) and the fact that
    complete filtered spectra are closed under colimits that we can view
    $\lim_{n}\tau_{\geq n}^\B X$ as a complete filtered
    spectrum $Y(\star)$ with graded pieces
    $\gr^iY\we\lim_{n}\gr^i\tau_{\geq n}^\B
    X\we\lim_{n}\tau_{\geq n-i}\gr^i X$.
    Hence, each $\gr^iY$ is $\infty$-connective. Thus, $\gr^iY\we 0$ for all
    $i$ and hence $Y(\star)\we 0$ as it is complete. This proves the
    completeness statement in (ii). Finally, $(\tau_{\leq n-1}^\B
    X)(0)\in\D(k)_{\leq n-1}$. It follows that $(\tau_{\geq n}^\B
    X)(0)\rightarrow X(0)\we X$ is an $n$-equivalence and exhaustiveness
    follows by letting $n\rightarrow-\infty$.
\end{proof}

\begin{figure}[h]
    \centering
    \includegraphics[scale=0.3]{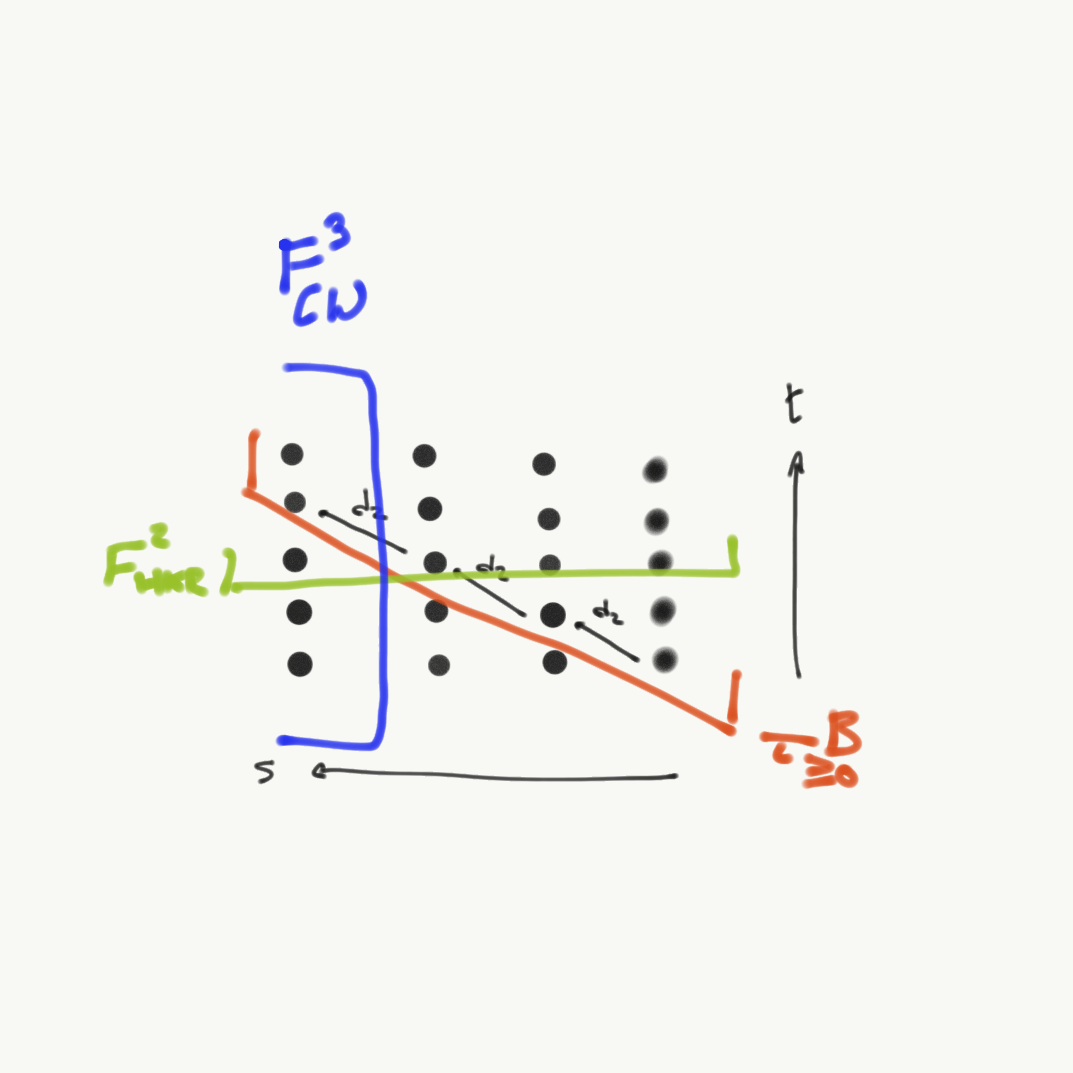}
    \caption{The Beilinson filtration. The figure shows the $\E_2$-page of the
    spectral sequence
    $\E_2^{s,t}=\H^s(BS^1,\HH_t(R/k))\Rightarrow\HC^-_{t-s}(R/k)$ and which parts of
    $\HC^-(R/k)$ are cut out by $\tau_{\geq 0}^\B\HC^-(R/k)$,
    $\F^2_{\HKR}\HC^-(R/k)$, and $\F^3_{\CW}\HC^-(R/k)$, respectively. For the
    definition of the CW filtration, see Section~\ref{sec:general}.}
    \label{fig:ss}
\end{figure}

\section{The general case}\label{sec:general}

Our general strategy for the proof of Theorem~\ref{thm:intro} is to left Kan extend
from the case of smooth algebras. Because of convergence issues, we are forced
to Kan extend in an $\infty$-category which keeps track of multiple
filtrations.

\newcommand{\poly}{\mathrm{poly}}

Let $k$ be a commutative ring, $\sCAlg_k$ the $\infty$-category of simplicial
$k$-algebras, and $\CAlg_k^{\poly}\subseteq\sCAlg_k$ the full subcategory of finitely
generated polynomial $k$-algebras. This embedding admits a universal property:
given any $\infty$-category $\Cscr$ which admits sifted colimits, the forgetful
functor $\Fun'(\sCAlg_k,\Cscr)\rightarrow\Fun(\CAlg_k^\poly,\Cscr)$ is an
equivalence, where $\Fun'(\sCAlg_k,\Cscr)$ is the $\infty$-category of sifted
colimit-preserving functors $\sCAlg_k\rightarrow\Cscr$. Given
$F:\CAlg_k^\poly\rightarrow\Cscr$, we call the corresponding sifted
colimit-preserving functor
$\mathrm{d}F\colon\sCAlg_k\rightarrow\Cscr$ the {\bf left Kan extension} or the {\bf non-abelian
derived functor} of $F$. For details, see~\cite{htt}*{Section~5.5.9}.

Let $R\in\sCAlg_k$ and fix $F\colon\CAlg_k^\poly\rightarrow\Cscr$.
Then, one extends $F$ to all polynomial rings by taking filtered colimits in
$\Cscr$. To compute the value of the left Kan extension $\mathrm{d}F$ of $F$ on
$R$, one takes a simplicial resolution $|P_\bullet|\we R$ where each
$P_\bullet$ is polynomial (but not necessarily finitely generated), and
computes $|F(P_\bullet)|$ in $\Cscr$.

Let $k$ be a commutative ring, and let $R$ be a simplicial commutative
$k$-algebra. Then, $\HH(R/k)$ is a connective commutative algebra object in
$\D(k)^{BS^1}$, the $\infty$-category of complexes of $k$-modules equipped with
$S^1$-action. We could apply Example~\ref{ex:circle} to obtain a filtration on
$\HC^-(R/k)=\HH(R/k)^{hS^1}$ with graded pieces truncations of the cochain
complex $(\HH_*(R/k),B)$. However, in the non-smooth case, this does not
capture derived de Rham cohomology.

We use the fact that Hochschild homology commutes with sifted colimits (see for
example~\cite{bms2}*{Remark~2.3})
to Kan extend the HKR filtration of~\cite{hochschild-kostant-rosenberg} from finitely generated polynomial algebras to
all simplicial commutative $k$-algebras. This
gives functorial complete exhaustive decreasing multiplicative $\NN$-indexed $S^1$-equivariant
multiplicative filtration $\F_{\HKR}^{\star}\HH(R/k)$ on $\HH(R/k)$ with graded pieces
$\gr^t_{\HKR}\HH(R/k)\we\Lambda^t\L_{R/k}[t]$ with the trivial $S^1$-action,
where $\L_{R/k}$ denotes the cotangent complex and $\Lambda^t\L_{R/k}$ is the
$t$th derived exterior power of the cotangent complex.
Since $\F_{\HKR}^t\HH(R/k)$ is $t$-connective for all $t$, it follows that the
HKR filtration is complete. 

Applying homotopy $S^1$-fixed points or Tate, we obtain decreasing
multiplicative $\NN$-indexed filtrations
$\F^\star_{\HKR}\HC^-(R/k)$ and $\F^\star_{\HKR}\HP(R/k)$ on negative cyclic
homology $$\HC^-(R/k)=\HH(R/k)^{hS^1}$$ and periodic cyclic homology
$$\HP(R/k)=\HH(R/k)^{tS^1}.$$ These filtrations are both
complete. To see this, note first that the induced HKR filtration
$\F^\star_{\HKR}\HC(R/k)$ on cyclic homology $\HC(R/k)=\HH(R/k)_{hS^1}$ is
complete since $\F^t_{\HKR}\HC(R/k)\we(\F^t_{\HKR}\HH(R/k))_{hS^1}$ is
$t$-connective. Thus, since we have a cofiber sequence
$\F^\star_{\HKR}\HC(R/k)[1]\rightarrow\F^\star_{\HKR}\HC^-(R/k)\rightarrow\F^\star_{\HKR}\HP(R/k)$
in $\DF(k)$, it suffices to see that the HKR filtration on
$\HC^-(R/k)$ is complete. But, this follows from the fact that $(-)^{hS^1}$
commutes with limits.


Negative cyclic homology admits a second filtration, coming from the
standard cell structure $\CC\PP^0\subseteq\CC\PP^1\subseteq\cdots$ on
$BS^1\we\CC\PP^\infty$. This second filtration is compatible with the HKR filtration
since on Hochschild homology the HKR filtration is $S^1$-equivariant.
To be precise, we consider the double filtration
$$\F^t_{\HKR}\F^s_{\CW}\HC^-(R/k)=\fib((\F^t_{\HKR}\HH(R/k))^{hS^1}\rightarrow(\F^t_{\HKR}\HH(R/k))^{h\Omega\CC\PP^{s-1}}),$$
which has graded pieces
$$\gr^t_{\HKR}\gr^s_{\CW}\HC^-(R/k)\we\Lambda^t\L_{R/k}[t-2s].$$
This bifiltration is multiplicative in the natural sense with respect to the
Day convolution symmetric monoidal structure on
$\Fun(\NN^\op\times\NN^\op,\D(k))$ where we give $\NN^\op\times\NN^\op$ the
symmetric monoidal structure coming from (the opposite of) addition in the
monoid $\NN\times\NN$.

We will let $\DBF(k)$ denote the $\infty$-category
$\Fun(\NN^\op\times\NN^\op,\D(k))$
of $\NN^\op\times\NN^\op$-indexed {\bf bifiltered complexes} of $k$-modules and we will denote by $\widehat{\DBF}(k)$
the full subcategory of $\DBF(k)$ on those {\bf bicomplete} bifiltered
complexes, i.e., those $X(\star,\star)$ such that for
each $s$ one has $\lim_tX(s,t)\we 0$ and for each $t$ one has $\lim_sX(s,t)\we 0$.
Note that either condition implies that $X(\star,\star)$ is complete in the
weaker sense that $\lim_{s,t}X(s,t)\we 0$.

\begin{remark}
    Bicomplete bifiltered objects are the same as complete filtered objects in
    the complete filtered derived category.
\end{remark}

\begin{lemma}
    For any simplicial commutative $k$-algebra $R$,
    the filtration $\F^\star_{\HKR}\F^\star_{\CW}\HC^-(R/k)$ is bicomplete.
\end{lemma}

\begin{proof}
    Fix $s$. We have $$\lim_t\F^t_{\HKR}\F^s_{\CW}\HC^-(R/k)\we 0$$ as both
    $(-)^{hS^1}$ and $(-)^{h\Omega\CC\PP^{s-1}}$ commute with limits. Now, fix
    $t$. Then, we want to show that
    $$\lim_s\F^t_{\HKR}\F^s_{\CW}\HC^-(R/k)\we\fib\left((\F^t_{\HKR}\HH(R/k)^{hS^1}\rightarrow\lim_s(\F^t_{\HKR}\HH(R/k))^{h\Omega\CC\PP^{s-1}}\right)\we
    0.$$ But, for any bounded below spectrum with an $S^1$-action $X$, the
    natural map $X^{hS^1}\rightarrow\lim_sX^{h\Omega\CC\PP^{s-1}}$ is an
    equivalence. Indeed, this follows by a computation if $X$ has a single
    non-zero homotopy group, and then it follows for all homologically bounded complexes by
    induction. Then, it follows in the limit up the Postnikov tower since both
    $(-)^{hS^1}$ and $\lim_s(-)^{h\Omega\CC\PP^{s-1}}$ commute with limits.
\end{proof}

We can Kan extend $\HC^-(-/k)$ with its
bifiltration from finitely generated polynomial $k$-algebras to all simplicial commutative
$k$-algebras to obtain $\F^\star_{\HKR}\F^\star_{\CW}\dHC^-(R/k)$, a
bifiltration on derived negative cyclic homology. Let $\widehat{\dHC}^-(R/k)$ denote bicompleted 
derived negative cyclic homology and let
$\F^\star_{\HKR}\F^\star_{\CW}\widehat{\dHC}^-(R/k)$ be the bicomplete bifiltration on
bicompleted derived negative cyclic homology, which is the Kan extension of
$\F^\star_{\HKR}\F^\star_{\CW}\HC^-(-/k)$ as a
functor $\CAlg_k^{\poly}\rightarrow\widehat{\DBF}(k)$ to all simplicial
commutative $k$-algebras.

\begin{lemma}\label{lem:comp}
    For any $R\in\sCAlg_k$,
    the natural map
    $\F^\star_{\HKR}\F^\star_{\CW}\widehat{\dHC}^-(R/k)\rightarrow\F^\star_{\HKR}\F^\star_{\CW}\HC^-(R/k)$
    is an equivalence in $\widehat{\DBF}(k)$.
\end{lemma}

\begin{proof}
    Since both bifiltered objects are bicomplete, it is enough to check on graded pieces.
    Since the graded pieces functors $\gr^t\gr^s\colon\widehat{\DBF}(k)\rightarrow\D(k)$
    commute with colimits,
    $\gr^t_{\HKR}\gr^s_{\CW}\widehat{\dHC}^-(R/k)$ is the left Kan
    extension of $R\mapsto\Omega_{R/k}^t[t-2s]$ from finitely generated
    polynomial algebras to all simplicial commutative $k$-algebras, which is
    precisely
    $\gr^t_{\HKR}\gr^s_{\CW}\HC^-(R/k)\we\Lambda^t\L_{R/k}[t-2s]$.
\end{proof}

\begin{remark}
    The lemma says that even though $\HC^-(-/k)$ does not commute with sifted
    colimits as a functor $\sCAlg_k\rightarrow\D(k)$, it does commute with
    sifted colimits as a functor $\sCAlg_k\rightarrow\widehat{\DBF}(k)$ when
    equipped with its skeletal and HKR filtrations. In particular, we can
    compute $\HC^-(R/k)$ by left Kan extending from finitely generated
    polynomial algebras and then bicompleting.
\end{remark}

Fix $s$ and consider the Whitehead tower
$$\cdots\rightarrow\tau_{\geq r}^\B\F_{\HKR}^\star\F^s_{\CW}\HC^-(R/k)\rightarrow\tau_{\geq
r-1}^\B\F_{\HKR}^\star\F^s_{\CW}\HC^-(R/k)\rightarrow\cdots$$ in the
Beilinson $t$-structure on filtered complexes, where we are taking Beilinson
connective covers in the $\HKR$-direction. Recall that
\begin{equation}\label{eq:1}\gr^t\tau_{\geq
r}^\B\F^\star_{\HKR}\F^s_{\CW}\HC^-(R/k)\we\tau_{\geq
-t+r}\gr^t_{\HKR}\F^s_{\CW}\TC^-(R/k)\we\tau_{\geq
-t+r}\fib\left((\Lambda^t\L_{R/k}[t])^{hS^1}\rightarrow(\Lambda^t\L_{R/K}[t])^{h\Omega\CC\PP^{s-1}}\right)
\end{equation}
and hence that
$$\gr^t\pi_r^\B\F^\star_{\HKR}\F^s_{\CW}\HC^-(R/k)\we\left(\pi_{-t+r}\fib\left(\Lambda^t\L_{R/k}[t])^{hS^1}\rightarrow(\Lambda^t\L_{R/k}[t])^{h\Omega\CC\PP^{s-1}}\right)\right)[-t+r].$$
Here, the notation implies that we view $\pi_{-t+r}$ of the object on the right
as a complex concentrated in degree $-t+r$.
If $R/k$ is smooth, we have
$\Lambda^t\L_{R/k}\we\Omega^t_{R/k}$. In particular, in this case, we see
that $$\gr^t\pi_r^\B\F^\star_{\HKR}\F^s_{\CW}\HC^-(R/k)\we\begin{cases}
    \Omega^t_{R/k}[-t+r] & \text{if $r$ is even and $r\leq 2t-2s$,}\\
    0&\text{otherwise.}
\end{cases}
$$

\begin{theorem}\label{thm:bi}
    There is a complete exhaustive multiplicative decreasing $\ZZ$-indexed filtration
    $\F^\star_{\B}$ on the bicomplete bifiltered complex
    $\F^\star_{\HKR}\F^\star_{\CW}\HC^-(R/k)$. The graded piece $\gr^u_{\B}\HC^-(R/k)$ is
    naturally equivalent to $\widehat{\L\Omega}^{\geq u}_{R/k}[2u]$, the
    Hodge-complete derived de Rham cohomology of $R$, naively truncated.
    Moreover, the remaining HKR and CW filtrations on $\gr^u_{\B}\HC^-(R/k)$
    both coincide with the Hodge filtration. Finally, the underlying filtration
    $\F^\star_{\B}\HC^-(R/k)$ in the sense of Remark~\ref{rem:underlying} is a complete filtration of $\HC^-(R/k)$; if
    $\L_{R/k}$ has Tor-amplitude contained in $[0,1]$, then the filtration is exhaustive.
\end{theorem}

\begin{proof}
    When $R/k$ is a finitely generated polynomial algebra,
    we take as our filtration $\F^\star_{\B}$ the double-speed Whitehead
    filtration $\tau_{\geq 2\star}^\B\F^\star_{\HKR}\F^s_{\CW}\HC^-(R/k)$ in the
    Beilinson $t$-structure.
    By definition of the Beilinson $t$-structure and the analysis in the paragraph above,
    $\pi_{2u}^\B\F^\star_{\HKR}\F^s_{\CW}\HC^-(R/k)$ is a chain complex of the form
    $$0\rightarrow\Omega^{u+s}_{R/k}\rightarrow\Omega^{u+s+1}_{R/k}\rightarrow\cdots,$$
    where $\Omega^{u+s}_{R/k}$ sits in homological degree $u-s$.
    Thus, as in Section~\ref{sec:smooth}, for $R$ smooth,
    $$\gr_\B^u\F^\star_{\HKR}\F^s_{\CW}\HC^-(R/k)\we\pi_{2u}^\B\F^\star_{\HKR}\F^s_{\CW}\HC^-(R/k)\we\Omega^{\bullet\geq {u+s}}_{R/k}[2u].$$
    Both the CW filtration and the HKR filtration induce the Hodge filtration
    on this graded piece.

    We claim that for $R/k$ a finitely generated polynomial algebra on $d$
    variables, for each $u$, the
    bifiltered spectrum $\F^u_{\B}\F^\star_{\HKR}\F^\star_{\CW}\HC^-(R/k)$ is
    bicomplete. For each $s$, this follows from Lemma~\ref{lem:basics}.
    In the other direction, as soon as $2s > 2d-u$, Equation~\ref{eq:1} shows
    that $\F^u_\B\F^\star_{\HKR}\F^s_{\CW}\HC^-(R/k)\we 0$, so completeness in the
    $\CW$-direction is immediate.

    We now view the filtration $\F^\star_\B$ as giving us a functor
    $\CAlg_k^{\mathrm{poly}}\rightarrow\Fun(\ZZ^\op,\widehat{\DBF}(k))$, which we left Kan
    extend to a functor $\sCAlg_k\rightarrow\Fun(\ZZ^\op,\widehat{\DBF}(k))$.
    We verify the necessary properties in a series of lemmas.

    \begin{lemma}
        For any $R\in\sCAlg_k$,
        $$\colim_{u\rightarrow-\infty}\F^u_\B\F^\star_{\HKR}\F^\star_{\CW}\HC^-(R/k)\we\F^\star_{\HKR}\F^\star_{\CW}\HC^-(R/k),$$
        where the colimit is computed in $\widehat{\DBF}(k)$.
    \end{lemma}

    \begin{proof}
        The colimit functor $\Fun(\ZZ^\op,\widehat{\DBF}(k))\rightarrow\widehat{\DBF}(k)$
        commutes with colimits, so this follows from Lemma~\ref{lem:comp} once we show
        that the filtration $\F^u_\B\F^\star_{\HKR}\F^\star\HC^-(R/k)$ is
        exhaustive on $\F^\star_{\HKR}\F^\star_{\CW}\HC^-(R/k)$ for $R$ finitely
        generated polynomial. This follows from Lemma~\ref{lem:basics}.
    \end{proof}

    \begin{lemma}\label{lem:completeness}
        We have
        $\lim_u\F^u_\B\F^\star_{\HKR}\F^\star_{\CW}\HC^-(R/k)\we 0$, where the
        limit is computed in $\widehat{\DBF}(k)$.
    \end{lemma}

    \begin{proof}
        By conservativity of the limit-preserving functors
        $\gr^t\gr^s\colon\widehat{\DBF}(k)\rightarrow\D(k)$, it is enough to see
        that $$\lim_u\gr^t\tau_{\geq
        2u}^\B\F^\star_{\HKR}\gr^s_{\CW}\HC^-(R/k)\we 0$$ for all
        pairs $(s,t)$. But, this object is $(2u-t)$-connective by definition of the
        Beilinson $t$-structure and because of the fact that colimits of
        $(2s-t)$-connective objects are $(2s-t)$-connective. Thus, the limit
        vanishes.
    \end{proof}

    \begin{lemma}
        The graded piece $\gr_\B^u\HC^-(R/k)$ is the bicomplete bifiltered
        object obtained by left Kan extending $R\mapsto\Omega^{\geq u}[2u]$ to all
        simplicial commutative rings, where the filtration is given by
        $\F^{(s,t)}\Omega^{\geq u}[2u]\we\Omega^{\geq u+\max(s-u,t-u,0)}[2u]$.
    \end{lemma}

    \begin{proof}
        Indeed,
        this is clear on finitely generated polynomial algebras by
        Section~\ref{sec:smooth} so this follows by Kan extension using the fact that
        $\gr^u\colon\Fun(\ZZ^\op,\widehat{\DBF}(k))\rightarrow\widehat{\DBF}(k)$
        commutes with colimits.
    \end{proof}

    Thus, we have proved the theorem except for the last sentence.
    Now, we examine the underlying filtration $\F^\star_\B\HC^-(R/k)$ on
    $\HC^-(R/k)$ given by forgetting the HKR and CW filtrations.

    \begin{lemma}\label{lem:mc}
        Let $\widehat{\DBF}(k)\rightarrow\D(k)$ be the functor that sends a
        bicomplete $\NN^\op\times\NN^\op$-index bifiltered spectrum $X(\star,\star)$ to $X(0,0)$.
        This functor preserves limits.
    \end{lemma}

    \begin{proof}
        It is the composition of
        the inclusion functor $\widehat{\DBF}(k)\rightarrow\DBF(k)$ (a right
        adjoint), and the limit preserving evaluation functor
        $X(\star,\star)\mapsto X(0,0)$ on $\DBF(k)$.
    \end{proof}

    From Lemmas~\ref{lem:completeness} and~\ref{lem:mc}, it follows that the filtration $\F^u_\B\HC^-(R/k)$ is a complete
    filtration on $\HC^-(R/k)$. Exhaustiveness is somewhat subtle.

    \begin{lemma}
        If $\L_{R/k}$ has Tor-amplitude contained in $[0,1]$, then the filtration $\F^\star_\B\HC^-(R/k)$ on
        $\HC^-(R/k)$ is exhaustive.
    \end{lemma}

    \begin{proof}
        Consider the cofiber $C^u$ of
        $\F^u_\B\HC^-(R/k)\rightarrow\HC^-(R/k)$ in $\widehat{\DBF}(k)$.
        We find that
        $$\gr^t_{\HKR}\gr^s_{\CW}\F^u_\B\HC^-(R/k)\we\begin{cases}0&\text{if
        $u>t-s$,}\\\Lambda^t\L_{R/k}[t-2s]&\text{otherwise.}\end{cases}$$
        Similarly,
        $\gr^t_{\HKR}\gr^s_{\CW}\HC^-(R/k)\we\Lambda^t\L_{R/k}[t-2s]$. It
        follows that
        $$\gr^t_{\HKR}\gr^s_{\CW}C^u\we\begin{cases}\Lambda^t\L_{R/k}[t-2s]&\text{if
        $t-s<u$,}\\0&\text{otherwise.}\end{cases}$$
        Since $\L_{R/k}$ has Tor-amplitude contained in $[0,1]$,
        $\Lambda^t\L_{R/k}$ has Tor-amplitude contained in $[0,t]$\footnote{Use that
        $\L_{R/k}$ is quasi-isomorphic to a complex $M_0\leftarrow M_1$
        where $M_0,M_1$ are flat, the fact that flats are
        filtered colimits of finitely generated projectives, the
        standard filtration on $\Lambda^t\L_{R/k}$ with graded pieces
        $\Lambda^jM_0\otimes_R\Lambda^{t-j}(M_1[1])$, and the fact that
        $\Lambda^{t-j}(M_1[1])\we(\Gamma^{t-j}M_1)[t-j]$, where $\Gamma^{t-j}$ is
        the divided power functor on flats.} and hence
        $\Lambda^t\L_{R/k}[t-2s]$ has Tor-amplitude contained in $[t-2s,2t-2s]$. In
        particular, we see that $C^u$ has a complete filtration with graded
        pieces having Tor-amplitude in $[t-2s,2t-2s]$ for $t-s<u$. In
        particular, since $R$ is discrete, the graded pieces are
        $2u$-coconnected. Since $C^u$ is a limit of $2u$-coconnected objects,
        it follows that $\pi_iC^u=0$ for $i\geq 2u$. In particular,
        $\colim_{u\rightarrow-\infty}C^u=0$ and the filtration is exhaustive as
        claimed.
    \end{proof}

    This completes the proof.
\end{proof}

Now, we give the argument for $\HP(R/k)$.

\begin{corollary}\label{cor:hp}
    There is a complete filtration $\F^\star_\B\HP(R/k)$ on $\HP(R/k)$ with
    $\gr^u_\B\HP(R/k)\we\widehat{\L\Omega}_{R/k}[2u]$. If $R/k$ is
    quasi-lci, the filtration is exhaustive.
\end{corollary}

\begin{proof}
    We use the cofiber sequence
    $\HC(R/k)[1]\rightarrow\HC^-(R/k)\rightarrow\HP(R/k)$. Note that
    $\HC(-/k)=\HH(R/k)_{hS^1}$ preserves colimits. The Kan extension of the HKR
    filtration on $\HC(-/k)[1]$ to from finitely generated polynomial $k$-algebras
    to all simplicial commutative $k$-algebras thus equips $\HC(-/k)[1]$ with an
    $\NN$-indexed filtration $\F^\star_{\HKR}\HC(-/k)[1]$  with graded pieces
    $\gr^n_\HKR\HC(-/k)[1]\we\Lambda^n\L_{R/k}[n+1]$. Moreover, since
    $\F^n_{\HKR}\HC(-/k)[1]$ is $n$-connective, the filtration is complete. By
    Lemma~\ref{lem:basics}, the double-speed Beilinson Whitehead tower induces
    a complete exhaustive decreasing $\ZZ$-indexed filtration
    $\F^\star_\B\HC(-/k)[1]$ on $\HC^-(-/k)[1]$. A straightforward check implies that the graded pieces are
    $\gr^u_\B\HC(-/k)[1]\we\L\Omega^{\leq u-1}_{R/k}[2u-1]$. Here, it makes no
    difference whether we take the Hodge-completed derived de Rham complex or
    the non-Hodge-completed derived de Rham complex, as the Hodge filtration on
    $\L\Omega^{\leq u-1}_{R/k}$ is finite. Now, we have a cofiber sequence
    $\F^\star_\B\HC(-/k)[1]\rightarrow\F^\star_\B\HC^-(R/k)\rightarrow\F^\star_\B\HP(R/k)$.
    Since the filtrations on $\HC(-/k)$ and $\HC^-(R/k)$ are complete, so is
    the induced filtration on $\HP(R/k)$. When $R/k$ is quasi-lci, Theorem~\ref{thm:bi}
    implies that the filtration on $\HC^-(R/k)$ is exhaustive. We have already
    noted that the filtration on $\HC(R/k)$ is exhaustive. Hence, the
    filtration on $\HP(R/k)$ is exhaustive. The graded pieces
    $\gr^u_\B\HP(R/k)$ fit into cofiber sequences $$\widehat{\L\Omega}^{\geq
    u}_{R/k}[2u]\rightarrow\gr^u_\B\HP(R/k)\rightarrow\L\Omega^{\leq
    u-1}_{R/k}[2u].$$ One finds using the remaining HKR filtration that in the
    smooth case $\gr^u_\B\HP(R/k)[-2u]$ is a chain complex (it is in the heart of
    the Beilinson $t$-structure) and that this sequence is equivalent to the
    canonical stupid filtration sequence
    $$0\rightarrow\Omega_{R/k}^{\bullet\geq
    u}\rightarrow\Omega_{R/k}^{\bullet}\rightarrow\Omega_{R/k}^{\bullet\leq
    u-1}\rightarrow 0.$$ This completes the proof since now we see in general
    that $\gr^u_\B\HP(R/k)\we\widehat{\L\Omega}_{R/k}[2u]$.
\end{proof}

\begin{proof}[Proof of Theorem~\ref{thm:intro}]
    Theorem~\ref{thm:bi} and Corollary~\ref{cor:hp} establish the theorem for
    affine $k$-schemes. It follows for general quasi-compact separated schemes
    because everything in sight is then computed from a finite limit of affine
    schemes, and the conditions of being complete or exhaustive are stable
    under finite limits. Finally, it follows for a quasi-compact quasi-separated
    scheme $X$ by induction on the number of affines needed to cover $X$.
\end{proof}


\vspace{20pt}
\scriptsize
\noindent
Benjamin Antieau\\
University of Illinois at Chicago\\
Department of Mathematics, Statistics, and Computer Science\\
851 South Morgan Street, Chicago, IL 60607\\
USA\\
\texttt{benjamin.antieau@gmail.com}
%


\addcontentsline{toc}{section}{References}

\begin{bibdiv}
\begin{biblist}

%
%

\bib{an1}{article}{
    author={Antieau, Benjamin},
    author={Nikolaus, Thomas},
    title={Cartier modules and cyclotomic spectra},
    note={In preparation},
    year={2018},
}

%
%
%
%
%
%
%
%
%

\bib{beilinson-perverse}{article}{
    author={Be\u\i linson, A. A.},
    title={On the derived category of perverse sheaves},
    conference={
    title={$K$-theory, arithmetic and geometry},
    address={Moscow},
    date={1984--1986},
    },
    book={
    series={Lecture Notes in
    Math.},
    volume={1289},
    publisher={Springer,
    Berlin},
    },
    date={1987},
    pages={27--41},
}

%
%
%
\bib{bms2}{article}{
    author={Bhatt, Bhargav},
    author={Morrow, Mathew},
    author={Scholze, Peter},
    title={Topological Hocschild homology and integral $p$-adic Hodge theory},
    journal = {ArXiv e-prints},
    eprint =  {http://arxiv.org/abs/1802.03261},
    year = {2018},
}
%
%
%
%
%
%
%
%
%
%
%
%
%
%
%
%
%
%
%
%
%
%
%

\bib{gwilliam-pavlov}{article}{
    author={Gwilliam, Owen},
    author={Pavlov, Dmitri},
    title={Enhancing the filtered derived category},
    journal={J. Pure Appl. Algebra},
    volume={222},
    date={2018},
    number={11},
    pages={3621--3674},
    issn={0022-4049},
}

%
%
\bib{hesselholt-ptypical}{article}{
    author={Hesselholt, Lars},
    title={On the $p$-typical curves in Quillen's $K$-theory},
    journal={Acta Math.},
    volume={177},
    date={1996},
    number={1},
    pages={1--53},
    issn={0001-5962},
}
%
%
%
\bib{hochschild-kostant-rosenberg}{article}{
    author={Hochschild, G.},
    author={Kostant, Bertram},
    author={Rosenberg, Alex},
    title={Differential forms on regular affine algebras},
    journal={Trans. Amer. Math. Soc.},
    volume={102},
    date={1962},
    pages={383--408},
    issn={0002-9947},
}
%
%
%
%
%
%
%
%
%
\bib{loday}{book}{
    author={Loday, Jean-Louis},
    title={Cyclic homology},
    series={Grundlehren der Mathematischen Wissenschaften},
    volume={301},
    note={Appendix E by Mar\'\i a O. Ronco},
    publisher={Springer-Verlag, Berlin},
    date={1992},
    pages={xviii+454},
    isbn={3-540-53339-7},
}
%
\bib{htt}{book}{
      author={Lurie, Jacob},
       title={Higher topos theory},
      series={Annals of Mathematics Studies},
   publisher={Princeton University Press},
     address={Princeton, NJ},
        date={2009},
      volume={170},
        ISBN={978-0-691-14049-0; 0-691-14049-9},
}
%
\bib{ha}{article}{
    author={Lurie, Jacob},
    title={Higher algebra},
    date={2017},
    eprint={http://www.math.harvard.edu/~lurie/},
    note={Version dated 18 September 2017},
}
 
%
%
%
%
%
%
%
%
%
%
%
%
%

\bib{toen-vezzosi-s1}{article}{
    author={To\"en, Bertrand},
    author={Vezzosi, Gabriele},
    title={Alg\`ebres simpliciales $S^1$-\'equivariantes, th\'eorie de de Rham et th\'eor\`emes HKR multiplicatifs},
    journal={Compos. Math.},
    volume={147},
    date={2011},
    number={6},
    pages={1979--2000},
    issn={0010-437X},
}

%

\end{biblist}
\end{bibdiv}
\end{document}